\documentclass[12pt]{article}
\usepackage{latexsym,amsmath,amssymb,color}
\usepackage{setspace}
\usepackage{amsthm}
\usepackage{eucal}
\usepackage{tikz}

\newcommand{\formal}[1]{\ensuremath{\textsf{#1}}}
\newcommand{\inside}{\ensuremath{\formal{in}}}
\newcommand{\outside}{\ensuremath{\formal{out}}}
\newcommand{\Left}{\ensuremath{\formal{left}}}
\newcommand{\Right}{\ensuremath{\formal{right}}}

\newcommand{\NP}{$\mathsf{NP}$}
\renewcommand{\P}{$\mathsf{P}$}
\newcommand{\WOne}{$\mathsf{W[1]}$}
\newcommand{\tw}{{\mathrm{tw}}}
\newcommand{\adj}{{\mathrm{adj}}}

\newcommand{\Integer}{\mathbb Z}

\newcommand{\Real}{\mathbb R}

\newtheorem{Theorem}{Theorem}

\newtheorem{Lemma}{Lemma}
\newtheorem{Corollary}{Corollary}
\theoremstyle{definition} \newtheorem{Definition}{Definition}
\theoremstyle{definition} \newtheorem{Remark}{Remark}
\theoremstyle{definition} \newtheorem{Algorithm}{Algorithm}
\theoremstyle{definition} 
\theoremstyle{definition} 
\newenvironment{Proof}{\begin{trivlist} \item[] {\bf Proof.}}{\hfill $\Box$\end{trivlist}}

\newcommand{\bigO}{\CMcal{O}}

\newcommand{\xshape}{\boxtimes}
\newcommand{\sF}{\mathcal{F}}
\newcommand{\sL}{\mathcal{L}}

\newcommand{\sP}{\mathcal{P}}
\newcommand{\sR}{\mathcal{R}}
\newcommand{\sT}{\mathcal{T}}

\newcommand{\cmP}{\CMcal{P}}

\newcommand{\ds}{\displaystyle}

\renewcommand{\geq}{\geqslant}
\renewcommand{\leq}{\leqslant}

\title{The {Firebreak} Problem}

\author{Kathleen D.~Barnetson\thanks{Department of Mathematics and Statistics, Memorial University of Newfoundland, St.~John's, NL, Canada.},
Andrea C.~Burgess\thanks{Department of Mathematics and Statistics, University of New Brunswick, Saint John, NB, Canada.  {\sl andrea.burgess@unb.ca}},
Jessica Enright\thanks{School of Computing Science, University of Glasgow, Glasgow, Scotland. {\sl Jessica.Enright@glasgow.ac.uk}},
\\
Jared Howell\thanks{School of Science and the Environment, Grenfell Campus, Memorial University of Newfoundland, Corner Brook, NL, Canada.  {\sl jahowell@grenfell.mun.ca}},
David A.~Pike\thanks{Department of Mathematics and Statistics, Memorial University of Newfoundland, St.~John's, NL, Canada.  {\sl dapike@mun.ca}},
Brady Ryan\thanks{Department of Mathematics and Statistics, Memorial University of Newfoundland, St.~John's, NL, Canada.}}

\begin{document}

\date{\today}

\maketitle

\begin{abstract}
Suppose we have a network that is represented by a graph $G$.  Potentially a fire (or other type of contagion) might erupt at some vertex of $G$.
We are able to respond to this outbreak by establishing a firebreak at $k$ other vertices of $G$, so that the fire cannot pass through these
fortified vertices.  The question that now arises is which $k$ vertices will result in the greatest number of vertices being saved from the fire,
assuming that the fire will spread to every vertex that is not fully behind the $k$ vertices of the firebreak.
This is the essence of the {\sc Firebreak} decision problem, which is the focus of this paper.
We establish that the problem is intractable on the class of split graphs as well as on the class of bipartite graphs,
but can be solved in linear time when restricted to graphs having constant-bounded treewidth,
or in polynomial time when restricted to intersection graphs.
We also consider some closely related problems.
\end{abstract}

\vspace*{\baselineskip}

\vspace*{\baselineskip}

\vfill\pagebreak
\section{Introduction}

In this paper we consider the {\sc Firebreak} decision problem, which asks whether it is possible to
establish a firebreak of a given size in a network represented by a graph and thereby protect a desired number of other vertices
from being reached by a fire that breaks out at a specified vertex of the graph.
The problem is formally stated as follows:

\begin{tabbing}
\hspace*{9mm}\=
{\textsc{Firebreak}}\\
\>\hspace*{5mm}\=
{Instance}: A graph $G$, an integer $k$, an integer $t$, and a vertex $v_f \in V(G)$.\\
\>\>
{Question}: \= Does $V(G)$ contain a $k$-subset $S$ such that $v_f \notin S$ and the number
\\
\>\>\>
of vertices of $G-S$ that are separated from $v_f$ is at least $t$?
\end{tabbing}

There are similarities between this problem and the well-known {\sc Firefighting} problem,
which itself takes the form of a game with two players (fire and firefighters).  The game begins
with fire starting at a vertex.  Thereafter, in each round of the game each firefighter is able to
designate one unburnt vertex as a permanent firebreak, and then the fire spreads from each
of its vertices to all of their unprotected neighbours.  The game concludes when the fire can spread no more
(which, in the case of infinite graphs, may result in a game that never terminates).
Depending on context, the goal of the firefighters may be to minimise the number of vertices that are scorched,
or to minimise the time in which the fire is contained.
The {\sc Firefighting} problem was introduced in 1995 by Hartnell~\cite{Hartnell1995} and has since attracted considerable attention.
For a survey of results and open questions, see~\cite{FM2009}.

The {\sc Firebreak} problem could be viewed as variant of the {\sc Firefighting} problem in which the
firefighters are active for only the initial round of the game, after which the fire spreads without further intervention.
The particular nature and formulation of the {\sc Firebreak} problem lends itself to several applications of practical interest.
Although we model the problem in terms of fire, it readily applies to the spread of any contagion from a point of infection in a network
and where a one-time response is able to be deployed in immediate reaction to the outbreak.

In the course of our investigation into the {\sc Firebreak} problem, we noted that it is also closely related
to what we will call the {\sc Key Player} decision problem that pertains to the number of connected components that can be
created by the removal of a set of vertices.
By defining $c(G)$ to be the number of connected components of a graph $G$
and $G-S$ to be the subgraph of $G$ that is induced by the vertices of $V(G) \setminus S$,
the {\sc Key Player} problem can be formally described as follows:


\begin{tabbing}
\hspace*{9mm}\=
{\textsc{Key Player}}\\
\>\hspace*{5mm}\=
{Instance}: A graph $G$, an integer $k$ and an integer $t$.\\
\>\>
{Question}: Does $V(G)$ contain a $k$-subset $S$ such that $c(G-S) \geq t$?
\end{tabbing}

The {\sc Key Player} problem also models various real-world scenarios and applications of practical interest in networks.
If we have the means to inoculate $k$ nodes which then become impenetrable to the contagion, we can ask which $k$ nodes to inoculate
in order to create the greatest number of segregated quarantine zones.
As another scenario, $G$ might represent a communications network that we wish to disrupt by selectively disabling $k$ of its nodes,
with the goal being to maximise the number of subnetworks that become unable to communicate with other subnetworks.

A related problem is one of dissemination rather than separation, whereby instead of selecting a $k$-set $S$ so that $c(G-S)$ is maximised,
we wish to selected a $k$-set $S$ that can most efficiently reach the other vertices of $G$.
One such scenario could be the spread of news:  we wish to directly inform $k$ individuals, who then propagate the news to their neighbours, so that
the news spreads to everybody as quickly as possible.
Instead of forwarding information, we might care about influence and opinion (such as might be the case in a marketing campaign,
whereby $k$ individuals are selected to receive a new commercial product in the hope that they will exert influence among their friends to
acquire the product).
This dissemination problem and its related separation problem appear to have been first jointly described by Borgatti in 2002~\cite{Borgatti2002}.
In~\cite{Borgatti2006} he refers to the dissemination scenario as the ``Key Player Problem / Positive'' and uses the phrase
``Key Player Problem / Negative'' for the problem involving the deletion of $k$ nodes.

This ``negative'' variant corresponds to our interest.
Some of the early papers about this problem in the literature attest to its applicability to network tolerance
and robustness~\cite{Ballester2006,OAH2008,Sathik2009}).
More recent results have considered it from a computational complexity perspective,
establishing that it is {\NP}-hard for various classes of graphs
yet being solvable in polynomial time for graphs having bounded treewidth~\cite{Addis2013,BGZ2014,Marx2006,ShenSmith2012}.
A recent survey by Lalou {\em et al.}\ on the topic of detecting critical nodes in networks also touches on this problem~\cite{Lalou2018}.
%
%
%
Incidentally, our literature search revealed that there is also
an edge-based version of the problem; although the edge version is beyond the scope of the present paper,
we nevertheless provide a few references for the interested reader (see~\cite{Baruah2012,KT1986,ProvanBall1983}).

In this paper we concentrate on the {\sc Firebreak} decision problem, while also presenting some new results about the {\sc Key Player} problem.
In Section~\ref{Sec-Intractability} we establish that both problems are {\NP}-complete when restricted to the class of split graphs.
We also find that the {\sc Firebreak} problem is {\NP}-complete when restricted to the class of bipartite graphs.
Although the {\sc Key Player} problem is {\NP}-complete (and it remains so for planar cubic graphs),
in the situation where the number of vertices to be removed from the graph coincides with the graph's connectivity
the problem is found to be solvable in polynomial time.
In Section~\ref{Sec-BoundedTreewidth} we consider graphs having treewidth that is bounded by a constant
and for such graphs we show that the {\sc Firebreak} problem can be solved in linear time,
and in Section~\ref{Sec-IntersectionGraphs} we show that it can be solved in polynomial time for some classes of intersection graphs.

Before continuing, we establish some basic notation and terminology.
For a graph $G=(V(G),E(G))$,
we let $\deg_G(u)$
(or just $\deg(u)$ if the graph $G$ is unambiguously implicit)
denote the degree of vertex $u$, and we use the notation $u \sim v$ to indicate that vertex $u$ is adjacent to vertex $v$.
The maximum degree of $G$ is $\Delta(G) = \max\{\deg(u) : u \in V(G)\}$.
By $N_G(v)$ (or just $N(v)$ if there is no ambiguity)
we denote the open neighbourhood of a vertex $v$, so that $N(v) = \{ u \in V(G) : u \sim v \}$;
the closed neighbourhood $N(v) \cup \{v\}$ will be denoted by $N[v]$.
The {\em order} of a graph $G$ is the cardinality $|V(G)|$ of its vertex set,
and its {\em size} is the cardinality $|E(G)|$ of its edge set.
If $A$ is a subset of $V(G)$ then the subgraph of $G$ induced by $A$, denoted $G[A]$, is the graph with vertex set $A$
and edge set $E(G) \cap \big\{ \{u,v\} : u,v \in A \big\}$.
Throughout this paper we limit ourselves to finite undirected graphs without loops and without parallel edges,
and so it follows that for the size of a graph $G$ on $n$ vertices we have $|E(G)| \leq \binom{n}{2} \in \bigO (n^2)$.

For an instance $(G,k,t,v_f)$ of the {\sc Firebreak} problem we define $\sF(G, k, v_f)$
as the maximum number of vertices of $G - S$ that are not in the same connected component as $v_f$,
where this maximum is taken over all choices for $k$-subsets $S \subseteq V(G) \setminus \{v_f\}$.
Any $k$-subset $S \subseteq V(G) \setminus \{v_f\}$ that separates $\sF(G, k, v_f)$ vertices from $v_f$ will be called an {\em optimal} set.

\section{Intractability}
\label{Sec-Intractability}

Clearly the {\sc Firebreak} and {\sc Key Player} decision problems are in {\NP}
since any given $k$-set $S$ can be easily validated to determine whether $S$ separates at least $t$ vertices from $v_f$
(for the {\sc Firebreak} problem)
or $c(G-S) \geq t$ (for the {\sc Key Player} problem).
In this section we show that these two decision problems are both {\NP}-complete.

A {\em split} graph is any graph $G$ that admits a vertex partition $V(G)=A \cup B$ such that $A \cap B = \emptyset$,
$G[A]$ is a maximum clique, and $B$ is an independent set ({\em i.e.}, $B$ is a set of pairwise non-adjacent vertices).
To prove that the {\sc Firebreak} and {\sc Key Player} decision problems are {\NP}-complete,
we will show that they are {\NP}-complete even when the input graph $G$ is restricted to the class of split graphs.
We first note that under certain conditions the problems can be easily solved.

\begin{Lemma} \label{FBP-easy}
The {\sc Firebreak} problem
can be solved in linear time when $|N(v_f)| \leq k$.
\end{Lemma}

\begin{Proof}
Let $(G,k,t,v_f)$ be an instance of the {\sc Firebreak} problem where $G$ is a
graph and $|N(v_f)| \leq k$.
We now choose $S$ to consist of every neighbour of $v_f$ plus any choice of $k - |N(v_f)|$ additional
vertices selected from $V(G) \setminus N[v_f]$.
Clearly no vertex in $V(G) \setminus (S \cup \{v_f\})$ will be in the same connected component of $G-S$ as $v_f$.
Since there are $|V(G)| - k - 1$ vertices in $G-(S \cup \{v_f\})$, the problem is equivalent to asking if $|V(G)| - k - 1 \geq t$.
Hence one simply needs to count the vertices in $G$, which can be done in linear time.
\end{Proof}

\begin{Lemma} \label{KPP-easy}
The {\sc Key Player} problem on a split graph $G$ can be solved in linear time when $k$ is at least the size $\omega(G)$ of a maximum clique in $G$.
\end{Lemma}

\begin{Proof}
Let $G$ be a split graph with $V(G) = A \cup B$, where $A$ induces a maximum clique and $B$ is an independent set,
and let $(G, k, t)$ constitute an instance of the {\sc Key Player} problem.
Observe that the maximum number of connected components will be produced by deleting $A$ and $k - |A|$ of the $|V(G)| - |A|$ vertices of $B$.
Thus the given instance of the {\sc Key Player} problem has an affirmative answer if and only if $|V(G)| - k \geq t$.
To solve this, one simply needs to count the vertices of $G$. This can clearly be done in linear time.
\end{Proof}

With the next result we show that when restricted to split graphs,
the {\sc Key Player} problem is no more difficult than the {\sc Firebreak} problem.

\begin{Lemma} \label{FBP-solves-KPP}
The {\sc Key Player} decision problem on split graphs can be solved in polynomial time with an oracle for the {\sc Firebreak} decision problem on split graphs.
\end{Lemma}

\begin{Proof}
Let $(G_1,k_1,t_1)$ constitute an instance of the {\sc Key Player} problem,
where $G_1=(A \cup B, E)$ is a split graph with a maximum clique on $A$ and independent set $B$.
Without loss of generality we may assume that $k_1 < |A|$ for otherwise the problem is easily solved by Lemma~\ref{KPP-easy}.

We proceed to formulate an instance $(G_2,k_2,t_2,v_f)$ of the {\sc Firebreak} problem as follows. Construct $G_2$ from $G_1$ by adding a new vertex named $v_f$ and adding an edge $\{u,v_f\}$ for each $u \in A$.
Let $k_2=k_1$ and $t_2=t_1-1$.

Now suppose there exists a $k_1$-subset $S_1$ of $V(G_1)$ such that $c(G_1-S_1) \geq t_1$. Since $G_1$ is a split graph, this means $G_1-S_1$ must have at least $t_2 = t_1 - 1$ isolated vertices, none of which are in $A$. Thus we let $S_2 = S_1$. Then the $t_2$ isolated vertices of $G_1-S_1$ are also isolated vertices of $G_2-S_2$. As none of these vertices are members of the clique induced by $A$, none of them are in the same connected component as $v_f$.
Thus $S_2$ is a $k_2$-subset of $V(G_2)$ such that $v_f \notin S_2$ and there are at least $t_2$ vertices in $G_2-S_2$ not in the same component as $v_f$.
Hence an affirmative answer to the {\sc Key Player} problem implies an affirmative answer to the associated {\sc Firebreak} problem $(G_2,k_2,t_2,v_f)$.

Conversely, suppose that the {\sc Firebreak} problem $(G_2,k_2,t_2,v_f)$ has an affirmative answer,
namely
a $k_2$-subset $S_2$ of $V(G_2) \setminus \{v_f\}$
such that there are at least $t_2$ vertices of $G_2-S_2$ that are not in the same connected component as $v_f$.
Since $v_f$ is part of the clique of $G_2$, these $t_2$ vertices must be isolated vertices in $G_2-S_2$. Let $S_1=S_2$. Then by the construction of $G_2$, these same $t_2$ vertices are also isolated in $G_1-S_1$ and so they form $t_2$ distinct components in $G_1-S_1$. Since by assumption $k_2 < |N(v_f)|$, then there must be at least one other vertex remaining in the clique with $v_f$.
This vertex also remains in the clique of $G_1$, so $c(G_1-S_1) = t_2 + 1 = t_1$
and hence the {\sc Key Player} problem has an affirmative solution.

Since the instance $(G_1,k_1,t_1)$ of the {\sc Key Player} problem has an affirmative answer if and only if
the associated instance $(G_2,k_2,t_2,v_f)$ of the {\sc Firebreak} problem has an affirmative answer,
and this associated instance can clearly be constructed in polynomial time,
then the {\sc Key Player} problem can be solved in polynomial time with the availability of an oracle for the {\sc Firebreak} problem.
\end{Proof}

We now proceed to show that the {\sc Firebreak} and {\sc Key Player} decision problems are both {\NP}-complete,
even when restricted to the class of split graphs.
To do so we will refer to the $t$-{\sc Way Vertex Cut} problem studied by Berger {\em et al.}~\cite{BGZ2014},
expressed as a decision problem as follows:

\begin{tabbing}
\hspace*{9mm}\=
{$t$-\textsc{Way Vertex Cut}}\\
\>\hspace*{5mm}\=
{Instance}: A graph $G$, an integer $k$ and an integer $t$.\\
\>\>
{Question}: Does $V(G)$ contain a subset $S$ such that $|S| \leq k$ and $c(G-S) \geq t$?
\end{tabbing}

\begin{Theorem}
\label{Thm-Intractable}
When restricted to split graphs,
the {\sc Firebreak} and {\sc Key Player} decision problems are both {\NP}-complete.
\end{Theorem}

\begin{Proof}
Relying on a construction of Marx~\cite{Marx2006} that is restricted to split graphs,
Berger {\em et al.}\ show that the $t$-{\sc Way Vertex Cut} problem is {\NP}-complete when restricted to split graphs~\cite{BGZ2014}.
By using an oracle for the {\sc Key Player} problem, it is straightforward to answer any given instance $(G,k,t)$ of the $t$-{\sc Way Vertex Cut} problem.
In particular, for each $k' \in \{1,2,\ldots,k\}$ present the {\sc Key Player} oracle with $(G,k',t)$.
The $t$-{\sc Way Vertex Cut} problem has an affirmative answer
if and only if one or more of the $k$ answers provided by the {\sc Key Player} oracle is affirmative.
Hence the {\sc Key Player} problem is {\NP}-complete when restricted to split graphs.
It then follows from Lemma \ref{FBP-solves-KPP} that the {\sc Firebreak} problem is also {\NP}-complete when restricted to split graphs.
\end{Proof}

\begin{Corollary}
The \textsc{Firebreak} and \textsc{Key Player} decision problems are {\WOne}-hard on split graphs with parameters $k$ and $t$.
\end{Corollary}

\begin{Proof}
For the {\sc Key Player} problem, the result has already been proved by Theorem~16 of~\cite{Marx2006}.
The reduction in Lemma~\ref{FBP-solves-KPP} is clearly polynomial in $k$ and $t$ and is thus parameterized in $k$ and $t$.
Thus the \textsc{Firebreak} problem also is {\WOne}-hard on split graphs with parameters $k$ and $t$.
\end{Proof}

We can also use split graphs to show that the {\sc Firebreak} problem is {\NP}-complete for bipartite graphs.

\begin{Theorem}
\label{Thm-Bipartite}
When restricted to bipartite graphs,
the {\sc Firebreak} decision problem is {\NP}-complete.
\end{Theorem}

\begin{Proof}
Suppose $(G, k, t, v_f)$ is an arbitrary instance of the {\sc Firebreak} problem on a split graph $G$,
where $V(G) = A \cup B$, $A$ induces a maximum clique in $G$, and $B$ is an independent set in $G$.
Without loss of generality we may assume that $k < |N(v_f)|$, as otherwise the problem is easily solved by Lemma~\ref{FBP-easy}.
We may further assume that $k < |A \cap N(v_f)|$ for otherwise an optimal set $S$ can be obtained by selecting each vertex of $A \cap N(v_f)$
along with $k - |A \cap N(v_f)|$ vertices of $V(G) \setminus ((A \cap N(v_f)) \cup \{v_f\})$
such that vertices of $B \cap N(v_f)$ are preferentially selected prior to selecting any other vertices of
$V(G) \setminus ((A \cap N(v_f)) \cup \{v_f\})$.

Let $G'$ be the graph obtained from $G$ by subdividing every edge of $G[A]$, so that
$|V(G')| = |V(G)| + \binom{|A|}{2}$ and $|E(G')| = |E(G)| + \binom{|A|}{2}$.
For convenience we let $C$ denote the set $V(G') \setminus V(G)$ of $\binom{|A|}{2}$ vertices that are newly created in this process,
and for any two distinct vertices $a_1,a_2 \in A$ let $c(a_1,a_2)$ be the vertex of $C$ that is a common neighbour of $a_1$ and $a_2$.
Observe that $G'$ is a bipartite graph with vertex bipartition $(A,B \cup C)$.
We will show that the instance $(G,k,t,v_f)$ of the {\sc Firebreak} problem has an affirmative answer if and only if
the instance $(G',k,  t+ \binom{ k }{2} , v_f)$
also has an affirmative answer.

First suppose that $S \subseteq (V(G) \setminus \{v_f\})$ is an optimal set for the instance $(G,k,t,v_f)$.
Since $k < |A \cap N_G(v_f)|$, in $G-S$ none of the vertices of $(A \setminus (S \cup \{v_f\}))$ are separated from $v_f$
and so if there should exist some vertex $u \in S \cap B$
then for each $v \in A \setminus (S \cup \{v_f\})$ the $k$-set $(S \setminus \{u\}) \cup \{v\}$ is also optimal.
From an iterated application of this observation it follows that there must exist an optimal set $S_0$ such that $S_0 \cap B = \emptyset$.
In $G'$, the number of vertices that are separated from $v_f$ by $S_0$ is $\sF(G,k,v_f) + \binom{|S_0 \cap A|}{2} = \sF(G,k,v_f) + \binom{k}{2}$
and hence $\sF(G',k,v_f) \geq \sF(G,k,v_f) + \binom{k}{2}$.

Now let $S' \subseteq (V(G') \setminus \{v_f\})$ be an optimal set for the instance $(G',k,  t+ \binom{ k }{2} , v_f)$
such that among all optimal sets, $S'$ has the least intersection with $C$.
Since $k < |A \cap N_G(v_f)|$, it follows that the set $A' = (A \setminus (S' \cup \{v_f\}))$ is not empty.
Recall that $v_f \in V(G)$, so either $v_f \in A$ or $v_f \in B$.

Consider the situation in which  $v_f \in B$.
If there should exist a vertex $y \in A'$ that is separated from $v_f$ in $G'-S'$,
then it is necessary that $c(y,z) \in S'$ for each $z \in (A \cap N_G(v_f)) \setminus S'$
and consequently $k = |S'| \geq |A \cap N_G(v_f) \cap S'| + |(A \cap N_G(v_f)) \setminus S'|$.
However, $k < |A \cap N_G(v_f)|$ and so none of the vertices of $A'$ are separated from $v_f$ in $G'-S'$.

If $v_f \in A$ then the set $A'$ equals $(A \cap N_G(v_f)) \setminus S'$,
which we partition into subsets $Y$ and $Z$
such that $Y$ consists of all vertices of $A'$ that are separated from $v_f$ in $G'-S'$,
and $Z$ consists of all vertices of $A'$ that are not separated from $v_f$ in $G'-S'$.
Since $k < |A \cap N_G(v_f)|$ then for some $a \in A'$ neither $a$ nor $c(a,v_f)$ is in $S'$, and hence $Z \neq \emptyset$.
For each vertex $y \in Y$ it is necessary that $\bigcup_{z \in Z \cup \{v_f\}} \{c(y,z)\} \subseteq S'$.
Consequently $k=|S'| \geq |A \cap N_G(v_f) \cap S'| + |Y|(1 + |Z|)$, which is at least $|A \cap N_G(v_f)|$ when $Y \neq \emptyset$
and thus it must be that $Y = \emptyset$.
Therefore, regardless of whether $v_f$ is in $A$ or $B$,
no vertex of $A'$ is separated from $v_f$ in $G'-S'$.

Let $w \in A'$.
If there should exist some vertex $x \in S' \cap C$
then the $k$-set $(S' \setminus \{x\}) \cup \{w\}$
would contradict the selection of the set $S'$ as an optimal set having minimum intersection with $C$.
Hence $S' \cap C = \emptyset$.
If there should exist some vertex $u \in S' \cap B$
then the $k$-set $(S' \setminus \{u\}) \cup \{w\}$ is also optimal.
Iterated application of this observation ensures that there is an optimal set $S''$ such that $S'' \cap B = S'' \cap C = \emptyset$.
In the graph $G$, the number of vertices that are separated from $v_f$ by $S''$
is $\sF(G',k,v_f) - \binom{|S'' \cap A|}{2} = \sF(G',k,v_f) - \binom{k}{2}$
and hence $\sF(G,k,v_f) \geq \sF(G',k,v_f) - \binom{k}{2}$.

So when $k < |A \cap N_G(v_f)|$ we conclude that
$\sF(G',k,v_f) = \sF(G,k,v_f) + \binom{k}{2}$.
Thus the {\sc Firebreak} instance $(G,k,t,v_f)$ for the split graph $G$ has an affirmative answer if and only if
the instance $(G',k,  t+ \binom{ k }{2},v_f)$
also has an affirmative answer.
\end{Proof}

\subsection{Some further comments about the {\sc Key Player} problem}

While the {\sc Firebreak} problem is the main focus of this paper, we do have some additional results
pertaining to the {\sc Key Player} problem.  In Theorem~\ref{Thm-Intractable} it was established that the
{\sc Key Player} problem is not only {\NP}-complete, but also that it remains so when restricted to the class of split graphs.
It happens that the problem is also {\NP}-complete when restricted to cubic planar graphs,
in contrast to the {\sc Firebreak} problem which we now show
can be solved in polynomial time when restricted to graphs of constant-bounded degree
(including cubic graphs).

\begin{Lemma}
The {\sc Firebreak} problem can be solved in polynomial time on graphs of constant-bounded degree.
\end{Lemma}

\begin{Proof}
Let $m$ be a fixed integer
and let $(G,k,t,v_f)$ be an instance of the {\sc Firebreak} problem where $G$ is a graph on $n$ vertices such that $\Delta(G) \leq m$.
If $k \geq \deg(v_f)$ then the answer is affirmative if and only if $t \geq n-1-k$
because separating the maximum number of vertices from $v_f$ is accomplished by deleting $N(v_f)$ plus $k-|N(v_f)|$ other vertices.
If $k < \deg(v_f)$ then the answer can be computed by exhaustively considering all $\binom{n}{k}$ $k$-subsets of $V(G) \setminus \{v_f\}$
and determining whether any of them separate at least $t$ vertices from $v_f$.
Since $k$ is bounded by the constant $m$, this computation can be done in polynomial time.
\end{Proof}

Showing that the {\sc Key Player} problem is intractable on cubic planar graphs
involves consideration of the well-known {\sc Independent Set} problem.

\begin{tabbing}
\hspace*{9mm}\=
{\textsc{Independent Set}}\\
\>\hspace*{5mm}\=
{Instance}: A graph $G$ and an integer $m$.\\
\>\>
{Question}: Does $G$ contain an independent set of at least $m$ vertices?
\end{tabbing}


\begin{Theorem}
The {\sc Key Player} decision problem on cubic planar graphs is {\NP}-complete.
\end{Theorem}

\begin{Proof}
%
We employ a reduction from the {\sc Independent Set} problem,
which Garey and Johnson established in 1977 to be {\NP}-complete on 3-regular planar graphs~\cite{GareyJohnson}.
Given an instance $(G,m)$ of the {\sc Independent Set} problem, construct an instance $(G,|V(G)|-m,m)$ of the {\sc Key Player} problem.
It is easy to see that {\sc Independent Set} has an affirmative answer if and only if this instance of the {\sc Key Player} problem has an affirmative answer. Thus an oracle for the {\sc Key Player} problem can be used to efficiently solve the {\sc Independent Set} problem.
\end{Proof}

Having established that the {\sc Key Player} problem is intractable for a variety of classes of graphs,
we now proceed to consider the special case in which the parameter $k$ is restricted to being the connectivity of the graph in question.
For any nontrivial graph $G$, its connectivity $\kappa(G)$ is the size of a smallest set $S$ of vertices such that $G-S$ is not a connected graph;
such a set $S$ is called a {\em cut} (or $k$-{\em cut} when we wish to explicitly mention the size of the cut).

As we shall see, when $k=\kappa(G)$ the {\sc Key Player} problem can be solved in polynomial time.
An initial thought for how to potentially prove that this is so is to enumerate all of the $\kappa(G)$-cuts of $G$
and if they are polynomial in number then simply calculate $c(G-S)$ for each $\kappa(G)$-cut $S$.
However, unlike edge cuts of size $\kappa'(G)$ (of which there are at most $\binom{n}{2}$;
see Section~4.3 of~\cite{NagamochiIbaraki2008} for a proof),
there can be exponentially many $\kappa(G)$-cuts in a graph $G$, as is the case with the graph illustrated in Figure~\ref{Fig-ManyCuts}.
Hence the na\"{\i}ve idea of examining each $\kappa(G)$-cut individually will not serve as a valid approach.

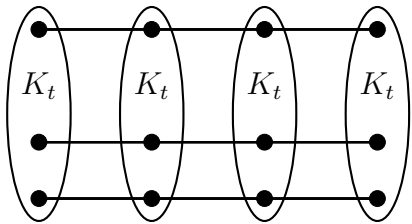
\begin{figure}[htbp]

\begin{center}
\begin{tikzpicture}[scale=.75]
\foreach \x in {0,2,4,6}
   {
   \draw [thick] (\x,1.5) ellipse (16pt and 54pt);
   \node at (\x,2) {$K_t$};
   \foreach \y in {0,1,3}
      {
      \draw [black,fill] (\x,\y) circle [radius=0.14];
      \draw [-, black, thick] (0,\y) -- (6,\y);
      }
   }

\end{tikzpicture}
\end{center}

\begin{quote}
\caption{A graph $G$ on $n=4t$ vertices.  $G$ consists of four copies of $K_t$ that are joined by $t$ disjoint paths of length 3.
The graph has connectivity $\kappa(G)=t$ and more than $2^t$ cut sets of size $\kappa(G)$.}
\end{quote}
\end{figure}
\label{Fig-ManyCuts}

Instead, we consider the notion of a $k$-{\em shredder} in a $k$-connected graph, which is defined to be
a set of $k$ vertices whose removal results in at least three components being disconnected from one another;
note that it is necessary here that $k \geq \kappa(G)$ since $G$ must be $k$-connected and each $k$-shredder is also a $k$-cut.
An algorithm that is capable of finding all of the $k$-shredders of a graph $G$ on $n$ vertices in
polynomial time is presented in~\cite{CheriyanThurimella1999}.
The actual number of $k$-shredders is determined in~\cite{Egawa2008} to be at most $\frac{2n}{3}$ when $k \geq 4$.
We can thus solve the {\sc Key Player} problem in polynomial time when $k = \kappa(G)$
by following the steps of the algorithm presented below.

\parbox{18cm}{
\begin{Algorithm}

~

\parbox{15cm}{
\begin{enumerate}
\item Let $k = \kappa(G)$.
\item If $k \geq 4$ then use the algorithm of~\cite{CheriyanThurimella1999} to find all of the $k$-shredders of $G$.
      As~\cite{Egawa2008} asserts that there are at most $\frac{2n}{3}$ of them, we then exhaustively check to see which $k$-shredder $S$ maximises $c(G-S)$.
      If there should happen to be no $k$-shredders, then it must be that every $k$-cut produces exactly two components.
\item If $k \leq 3$ then exhaustively check each $k$-subset $S$ of $V(G)$ to see which $k$-sets are cut sets,
      and then determine which $k$-cut maximises $c(G-S)$.
      The number of $k$-subsets that must be checked is $\binom{n}{k}$, which is polynomial in $n$ since $k$ is either 1, 2 or 3.
\item Having determined the maximum number of components that can result from the deletion of any $k$-cut,
      compare this quantity with $t$ to answer the given instance of the {\sc Key Player} problem.
\end{enumerate}
}

\end{Algorithm}
}

\section{Graphs with Constant-Bounded Treewidth}
\label{Sec-BoundedTreewidth}

Before we review the technical details of treewidth, we first observe that the {\sc Firebreak} problem can be easily solved
in the case where the graph in question is a tree.

\begin{Theorem}
\label{Thm-FirebreakTrees}
The {\sc Firebreak} problem can be solved in polynomial time on trees.
\end{Theorem}

\begin{Proof}
Suppose $(T,k,t,v_f)$ is an instance of the {\sc Firebreak} problem where $T$ is a tree on $n$ vertices.
If $|N(v_f)| \leq k$ then a polynomial time solution follows from Lemma~\ref{FBP-easy},
so we henceforth assume that $|N(v_f)| > k$.

Root the tree $T$ at $v_f$.
For each vertex $v$ of $T$ define $T(v)$ to be the subtree of $T$ rooted at $v$
and let $\sT(v)$ denote the number of vertices in $T(v)$.
If $v$ is a leaf in the tree then clearly $\sT(v) = 1$.
For any other vertex $v$ with children $x_1, \dotsc , x_m$, we have that $\sT(v) = \sT(x_1) + \dotsb + \sT(x_m)+1$.
The computation of each $\sT(v)$ can be clearly done in $\bigO (n)$ time.

To find an optimal $k$-set $S$ (namely one that separates the most vertices from $v_f$)
simply select  $k$ vertices in $N(v_f)$ having the $k$ greatest subtree sizes.
If $v_1, \dotsc , v_{k}$ are these $k$ vertices, then there are $\sT(v_1) + \dotsb + \sT(v_{k}) - k$ vertices not in the same connected component of $v_f$.
%
%

To answer the given instance of the {\sc Firebreak} problem, it now suffices to ask if
$\sT(v_1) + \dotsb + \sT(v_{k}) - k \geq t$.
Since both the computation of each $\sT(v)$ and the identification of the $k$ largest values of $\sT(v)$ can be done in polynomial time,
the problem can therefore be answered in polynomial time.
\end{Proof}

The proof of Theorem~\ref{Thm-FirebreakTrees} comes close to proving that the {\sc Firebreak} problem can be solved in linear $\bigO (n)$ time for a tree on $n$
vertices, except that the task of selecting the $k$ neighbours of $v_f$ with the greatest subtree sizes may require a nonlinear sort to be performed.
However, it will be shown later in this section that a linear time solution does nevertheless exist.
Our approach will be to consider the effect that bounded treewidth has on the complexity of the problem.
The treewidth parameter, defined below, was introduced by Robertson and Seymour~\cite{RS1986}.

\begin{Definition}
A \textit{tree decomposition} of a graph $G$ is a pair $(X,T=(I,F))$ where $X = \{X_i : i \in I\}$ is a family of subsets of $V(G)$, and $T$ is a tree whose vertices are the subsets $X_i$ such that:

\begin{enumerate}
\item $\bigcup_{i \in I} X_i = V(G)$.
\item For every edge $uv \in E(G)$, both $u$ and $v$ are in some $X_i$, $i \in I$.
\item If $i, j, k$ are vertices of $T$, and $k$ lies on the (unique) path from $i$ to $j$, then $X_i \cap X_j \subseteq X_k$.
\end{enumerate}
The \textit{width} of a tree decomposition is max$\{|X_i| - 1 : i \in I \}$.
The \textit{treewidth} $\tw(G)$ of a graph $G$ is the minimum width of all tree decompositions of $G$.
\end{Definition}

It is easy to see that trees have treewidth at most 1.
Other graphs with small treewidth are, in a sense, tree-like.
For instance,
if the treewidth of a graph $G$ is bounded by a constant ({\em i.e.}, $\tw(G) \leq w$),
then it follows from Lemma~3.2 of~\cite{RS1995} that $|E(G)| = \bigO (n)$ where $n = |V(G)|$.
Graphs that are in some way similar to trees often lend themselves to tractable solutions
for problems that are intractable for graphs in general
(see~\cite{Arnborg1985,AP1989,Bodlaender1987} for details of several examples).

Moreover,
Bodlaender has presented an algorithm that finds a tree decomposition of a graph $G$ in time that is linear in the number of vertices
and exponential in the cube of the treewidth~\cite{Bodlaender1996}.
For graphs having constant-bounded treewidth,
it is therefore possible to find tree decompositions in linear time.

Theorem~\ref{Thm-FirebreakTrees} demonstrated that the {\sc Firebreak} problem is easily solved for trees.
To show that it
is also tractable for graphs for which the treewidth is bounded by a constant,
we will rely on a powerful result that is based on work of Courcelle, independently proved by Borie, Parker and Tovey,
and further extended by Arnborg, Lagergren and Seese~\cite{ALS1991,BPT1992,Courcelle1990}.
A survey by Langer {\em et al.}~\cite{LRSS2014} presents it in a slightly more general form than we require.
For our purposes, the following will suffice:

\begin{Theorem}[see Theorem~30 of~\cite{LRSS2014}]
\label{Thm-Courcelle}
Let $G$ be a graph on $n$ vertices, let $w$ be a constant,
and let $\sP$ be a graph theoretic decision problem that can be expressed in the form of
extended monadic second-order logic.
If $\tw(G) \leq w$
then determining whether $G$ has property $\sP$ can be accomplished in time
$\bigO(f_{\sP}(w) \cdot n)$
where $f_{\sP}$ is a function that depends on the property $\sP$.
\end{Theorem}

Although the function $f_{\sP}$ may not be polynomial, if $w$ is a constant then so too is $f_{\sP}(w)$.
Hence decision problems that have extended monadic second-order (EMSO) formulations
are fixed-parameter tractable.
Monadic second-order (MSO) logic expressions for graphs are based on
\begin{itemize}
\setlength{\itemsep}{-0.20\baselineskip}
\item variables for vertices, edges, sets of vertices and sets of edges,
\item universal and existential quantifiers,
\item logical connectives of conjunction, disjunction and negation,
\item and binary relations to assess set membership, adjacency of vertices,
      incidence of edges and vertices, and equality for vertices, edges and sets.
\end{itemize}
We will only need to consider vertices and sets thereof, which will be respectively denoted by lower case and upper case variable names.
Predicates can be constructed from the basic ones and incorporated into expressions (in this manner a predicate for implication can be built).
To provide an illustrative example,
the expression
\begin{center}
$\ds \exists X \exists Y \,
\big(\forall u \, ((u \in X) \wedge (u \not\in Y)) \vee ((u \not\in X) \wedge (u \in Y)) \big)
\wedge
\linebreak
\big(\forall u \forall v \, ((\adj(u,v)) \Rightarrow ((u \in X) \wedge (v \in Y)) \vee ((u \in Y) \wedge (v \in X))   ) \big)
$
\end{center}
encodes whether a given graph is bipartite, where $\adj(u,v)$ represents a Boolean predicate that evaluates whether vertices $u$ and $v$ are adjacent.

Extended MSO logic has additional features that enable set cardinalities to be considered.
The survey by Langer {\em et al.}~\cite{LRSS2014} provides an excellent overview, to which we direct readers for more details.


Since the factor $f_{\sP}(w)$ in the conclusion of Theorem~\ref{Thm-Courcelle} effectively becomes a hidden constant,
it follows that
deciding whether a graph $G$ has the property $\sP$ can be done in linear time
when the hypothesis of the theorem is satisfied.
With this in mind,
we now show that the {\sc Firebreak} problem is tractable when restricted to graphs having treewidth at most a constant $w$.

\begin{Theorem}
The {\sc Firebreak} problem can be solved in linear time for graphs with constant-bounded treewidth.
\end{Theorem}

\begin{proof}
Let $\varphi$ represent the following logic expression with two
set variables ($S$ and $X$).

\begin{center}
$
\ds
\varphi =
( v_f \not\in S)
\wedge
(v_f \not\in X)
\wedge
\big( \forall y \, (y \in S) \Rightarrow (y \not\in X) \big)
\wedge
\linebreak
\big( \forall x \forall y \, \big( (x \in X) \wedge (\adj(x,y)) \wedge (y \not\in S) \big) \Rightarrow  (y \in X)   \big)
$
\end{center}
Observe that $\varphi$ encodes whether the set $S$ separates the set $X$ from a designated vertex $v_f$.
To take into consideration the cardinalities of the sets $S$ and $X$,
we now describe how to construct an evaluation relation $\psi$ as indicated by Definition~18 of~\cite{LRSS2014}.
Following the notation of~\cite{LRSS2014}, given that we have two
sets ($S$ and $X$) and two integer input values ($k$ and $t$),
choosing $m=1$ will result in $\psi$ having four variables:
\begin{center}
$\ds
y_1 = \sum_{u \in S} w_1(u)
\hspace*{11mm}
y_2 = \sum_{u \in X} w_1(u)
\hspace*{11mm}
y_3 = k
\hspace*{11mm}
y_4 = t
$
\end{center}
Define the weight function $w_1 : V(G) \rightarrow \Real$ such that $w_1(v) = 1$ for each $v \in V(G)$.
Now, let $\psi$ be the evaluation relation
$$
(y_1 = y_3) \wedge (y_2 \geq y_4)
$$
We have adhered to Definition~18 of~\cite{LRSS2014}.
Hence we have created an EMSO expression
that encodes the {\sc Firebreak} decision problem.
The result now follows from Theorem~\ref{Thm-Courcelle}.
\end{proof}

\section{Intersection Graphs}
\label{Sec-IntersectionGraphs}
Given a family of sets $\mathcal{S} = \{S_0, S_1, \ldots, S_k\}$, the \emph{intersection graph} of $\mathcal{S}$ is a graph $G= (V, E)$ for which there exists a bijection $f$ between $V$  and $\mathcal{S}$ such that $u$ is adjacent to $v$ in $G$ if and only if $f(u) \cap f(v) \neq \emptyset$, that is, $f(u)$ \emph{intersects} $f(v)$.  We say that the bijective assignment and the family of sets are a \emph{representation} of $G$.  When we restrict the nature of the representing sets, we can restrict the class of representable graphs, and structured representations have provided a wide variety of tractability results (many examples are listed in \cite{golumbic}).

Here, we give polynomial-time algorithms for two classes of intersection graphs: subtree intersection graphs of limited leafage and permutation graphs.  In both cases, we use an approach that sweeps the representation for separators, allowing us to exhaustively check these separators for firebreak feasibility.

\subsection{Intersection graphs of subtrees in a tree}

In this section we focus on the intersection graphs of subtrees in a tree of constant bounded leafage, for which we show that the {\sc Firebreak} problem can be solved in polynomial time.   For our purposes, two subtrees are considered to intersect if they share at least one vertex.  The intersection graphs of a tree are the chordal graphs, and the intersection graphs of trees with a constant bounded number of leaves (the \emph{leafage}) can be recognised and a representation constructed in time $\bigO (n^3)$~\cite{HabibStacho} (where $n$ is the number of vertices in the graph to be represented). Because these are a subfamily of chordal graphs, for which $\tw(G) = \omega(G) - 1$, they do not in general have constant-bounded treewidth and so the results from Section~\ref{Sec-BoundedTreewidth} do not apply to them.




\begin{Theorem}
\label{Thm-Intersection}
The {\sc Firebreak} problem on a graph $G= (V, E)$ that is the intersection graph of subtrees of a tree of leafage $\ell$ can be solved in time $\bigO (|V|^{2\ell +1})$.
\end{Theorem}
\begin{proof}

Let $(G,k,t,v_f)$ be an instance of the {\sc Firebreak} problem where $G$ is the intersection graph of subtrees $\mathcal{T}$ of a tree $T = (V_T, E_T)$ with constant bounded number of leaves $\ell$, and denote by $\mathcal{T}(v)$ the subtree of $T$ that represents vertex $v \in V(G)$.  For convenience, let $n = |V|$. Recall that we denote by $S$ the set of $k$ vertices that we remove from the graph $G$ in order to form a firebreak.

We make a simplifying assumption that we should not place in $S$ a neighbour of $v_f$ that is adjacent only to other neighbours of $v_f$,
as such a vertex being included in $S$ can only ever protect from burning that single vertex and no others;
there is always a non-worse choice of vertex for inclusion in $S$.
Note also that we assume that $v_f$ has more than $k$ neighbours, for if not then we apply Lemma~\ref{FBP-easy} to resolve the question in $\bigO(n)$ time.

We argue that we can find a polynomially-bounded number of useful minimal separators, that any solution to the {\sc Firebreak} problem will place in $S$ at most a constant bounded number of them, and that we can check each candidate set of separators efficiently.

Given the representation $\mathcal{T}$, which by \cite{HabibStacho} we can construct in time $\bigO(n^3)$, we know from \cite{kumar} that there are at most  $\bigO(n^2)$ minimal vertex separators in our graph $G$, and that they correspond to the vertices of $T$: specifically, there is one for each vertex $u$ of $T$, and it is composed of the vertices of $G$ that are represented by subtrees that contain $u$.  Note that (also due to \cite{kumar}) there are at most  $\bigO(n^2)$ vertices in $T$.  
Any candidate separating set $S$ that will serve as a certificate to an affirmative answer to our problem instance $(G,k,t,v_f)$
must be composed of the union of a set of these minimal separators.



There is a unique path from each leaf to the closest vertex of $\mathcal{T}(v_f)$.  Let $v_i, v_j$ be vertices on that path in the tree, and let $S_i, S_j$ be their corresponding minimal vertex separators in $G$.  Without loss of generality, let $v_i$ be closer to $\mathcal{T}(v_f)$ than $v_j$ is.  Then the set of vertices separated from $v_f$ in $G - S_i$ is at least as large as the set of vertices separated from $v_f$ in $G - (S_i \cup S_j)$.  Thus in a candidate separating set $S$ we need include only at most one minimal separator corresponding to a vertex on the unique path from each leaf to the subtree $\mathcal{T}(v_f)$.  There are at most $\ell$ such paths, so there are ${n^2\choose \ell} = \bigO(n^{2\ell})$ possible combinations of minimal separators to consider when constructing candidate solutions to our problem instance $(G,k,t,v_f)$.

Given a particular candidate separator $S$, we can check if it provides a certificate for an affirmative answer to $(G,k,t,v_f)$ by checking to see if $|S| \leq k$, and if the number of vertices separated from $v_f$ in $G - S$ is at least $t$ in $\bigO(n)$ time.

Thus we can generate all candidate solutions in time $\bigO(n^{2\ell})$, and check the feasibility of each in $\bigO(n)$, giving an overall running time of $\bigO(n^{2\ell+1})$.
\end{proof}

As a special case for leafage $\ell=2$ this argument gives us an algorithm to solve the {\sc Firebreak} problem in interval graphs in time $\bigO (n^{5})$.  We can do somewhat better in this case using a representation-construction algorithm due to Booth and Lueker~\cite{BoothLueker1976}, who give an $\bigO (|V| + |E|)$ algorithm, which, using the reasoning above, we can use to give an $\bigO ((|V| + |E|)^4)$ algorithm.

\subsection{Permutation Graphs}

There are a large variety of types of intersection graphs (in fact, every graph is an intersection graph of some set of objects).
While Theorem~\ref{Thm-Intersection} applies to intersection graphs of paths in a tree (which we note includes interval graphs),
permutation graphs are not a subclass of this class of intersection graphs.
By using a
sweeping-for-separators approach we are able to show that the
{\sc Firebreak} problem is tractable on permutation graphs as well.
We make use of one of the many characterisations of a permutation graph, as in~\cite{Spinrad1985}:  a \emph{permutation graph} is the intersection graph of straight line segments between two parallel lines.

We give an example of a permutation graph and a corresponding representation in Figure~\ref{fig:permutationExample}.   We will make use of several results on these representations to address the {\sc Firebreak} problem on permutation graphs.

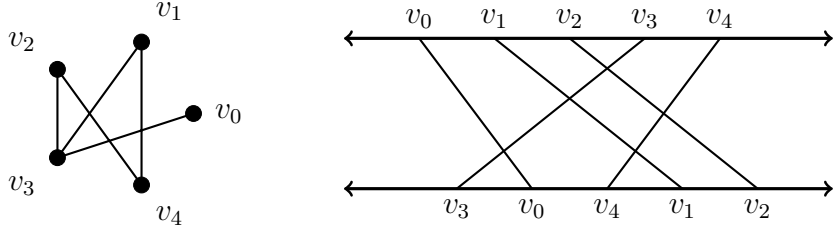
\begin{figure}

\begin{center}
\begin{tikzpicture}[scale=1, auto=center]

\foreach \x/\y in {0/1.5,1/3.5,2/4.5,3/.5,4/2.5}{
\coordinate[label={above:$v_{\x}$}] (v\x) at (\x,2) {};
\coordinate[label={below:$v_{\x}$}] (u\x) at (\y,0) {};
\draw[thick] (v\x)--(u\x);}

\draw[very thick,<->] (-1,0)--(5.5,0);
\draw[very thick,<->] (-1,2)--(5.5,2);

\begin{scope}[xshift=-4cm,yshift=1cm]
\foreach \x in {0,1,2,3,4}
{
\node[label={{\x*72}:$v_{\x}$}] (w\x) at ({\x*72}:1) {};
\coordinate (p\x) at ({\x*72}:1);
\draw [black,fill] ({\x*72}:1) circle [radius=0.105];
}
\end{scope}

\draw[thick] (p0)--(p3)--(p1)--(p4)--(p2)--(p3);

\end{tikzpicture}
\end{center}

\caption{An example of a permutation graph and an associated representation.}
\label{fig:permutationExample}
\end{figure}

\begin{Theorem}
\label{Thm-Permutation}
An instance $(G,k,t,v_f)$ of the {\sc Firebreak} problem where $G$ is a permutation graph on $n$ vertices
can be solved in time $\bigO (n^3 k^2)$.
\end{Theorem}

\begin{proof}

As noted by Lemma~\ref{FBP-easy}, the case in which $|N(v_f)|\leq k$ can be solved in $\bigO(n)$ time, so we will now consider the case where $|N(v_f)| > k$. A permutation graph has a \emph{representation} in the form of line segments between two parallel horizontal lines, which itself can be created in $\bigO(n^2)$ time, and in which we can assume without loss of generality that there are no vertical line segments~\cite{Spinrad1985}.

This representation gives a partition of vertices: those corresponding to line segments to the left of the line segment representing $v_f$, those corresponding to line segments to the right of the line segment representing $v_f$, and those adjacent to $v_f$ ({\em i.e.}, the lines that intersect the line representing $v_f$).	 
In this representation there are $n$ line-segment endpoints on the top horizontal line, and $n$ on the bottom (there is one of each for each vertex).  If we segment the horizontal lines into portions \emph{between} line segment endpoints, there are therefore $n-1$ segments of each.
As described in~\cite{BKK1995}, we can find all minimal separators in a permutation graph from the line segment representation by considering a series of cut-lines, where a \emph{cut-line} is a line between the two horizontal walls with one endpoint on the top horizontal line and one on the bottom, and we consider one cut-line for each pair of top segment and bottom segment.  There are thus $\bigO(n^2)$ such cut-lines, and again due to~\cite{BKK1995} we know that every minimal separator in the permutation graph consists of the vertices corresponding to the line segments crossed by one of these cut-lines in the representation.

Of these $\bigO(n^2)$ minimal separators, there are $\bigO(nk)$ that are of interest to us,
namely those that contain at most $(k-1)$ vertices; for convenience we call minimal vertex separators that meet this size constraint \emph{$(k-1)$-small}.
Given any $(k-1)$-small minimal separator $S$ defined by the cut-line $s$, let $S_\Left$ be the subset of vertices to the left of $s$
and let $S_\Right$ be the subset of vertices to the right of $s$.

Consider the following algorithm that searches for firebreaks
({\emph{i.e.}, $k$-subsets of $V(G)$ that separate $v_f$ from some other vertices})
of the form $S \cup T$.

\parbox{18cm}{
\begin{Algorithm}

~

\parbox{15cm}{
\begin{enumerate}
\item
 Find all minimal separators $S$ with $v_f\in S_\Left$ (resp.\ $S_\Right$) which have size at most $k$,
 and denote this set as $\sL$ (resp.\ $\sR$).


\item Exhaustively search for all pairs $(S,T)$ such that $S \in \sL$, $T \in \sR$, $|S|+|T|\leq k$ and $|S_\Right|+|T_\Left|\geq t$.

      If such a pair $(S,T)$ is disjoint and $|S|+|T|=k$, then $S\cup T$ is a firebreak.


      Otherwise,
      consider the component $C$ in $G-(S\cup T)$ containing $v_f$. If $|V(C)|-1+|S_\Right|+|T_\Left|-t+|S\cup T|\geq k$ then there is
      a firebreak as $|V(C)|-1$ vertices can be removed from this component and $|S_\Right|+|T_\Left|-t$ vertices can be removed from other components to produce a firebreak of size $k$ by adding them to $S \cup T$.

\item If no firebreak was found during the exhaustive search, then no firebreak exists.
\end{enumerate}
}

\end{Algorithm}
}

Suppose some firebreak exists but this algorithm found none. There exist $t$ vertices that can be separated from $v_f$, possibly some to the left, say $L$, and some to the right, say $R$, of $v_f$. Note that at most one of $L$ and $R$ can be empty. Since $L$ and $R$ are separated from $v_f$, there must be cut-lines between them and $v_f$ that define minimal separators and hence the algorithm must have found a firebreak.

Since there are $\bigO(nk)$ minimal separators that are $(k-1)$-small, we can search through the pairs in $\bigO(n^2k^2)$ time. For each pair $(S,T)$ we may have to find  $G-(S\cup T)$ and count $|V(C)|$, $|S_\Right|$, and $|T_\Left|$. This adds a factor of $n$ and thus the problem can be solved in $\bigO(n^3k^2)$ time.
%
\end{proof}

\section{Acknowledgements}
Authors Burgess and Pike acknowledge research grant support from NSERC.
Ryan acknowledges support from an NSERC Undergraduate Student Research Award.


\end{document}